\documentclass[11pt]{amsart}

\usepackage[utf8]{inputenc}
\usepackage{amsmath,amssymb,amsthm,mathtools}
\usepackage{enumitem}
\usepackage{tikz}
\usepackage[hidelinks]{hyperref}
\newtheorem{theorem}{Theorem}
\newtheorem{lemma}[theorem]{Lemma}
\newtheorem*{factA}{Fact A}
\newtheorem*{factB}{Fact B}
\theoremstyle{definition}
\newtheorem{definition}[theorem]{Definition}

\newcommand{\R}{\mathbb{R}}
\newcommand{\Q}{\mathbb{Q}}
\newcommand{\N}{\mathbb{N}}
\newcommand{\Z}{\mathbb{Z}}
\newcommand{\Cs}{2^{\N}}
\newcommand{\cR}{\mathcal{R}}
\newcommand{\cB}{\mathcal{B}}
\newcommand{\ip}[2]{\langle #1,#2\rangle}
\newcommand{\Fm}{F_m}
\newcommand{\Fl}{F_m^{-}}
\newcommand{\Fr}{F_m^{+}}

\begin{document}

\title[A Borel set with non-Borel projections]{A bounded Borel set in the plane none of whose orthogonal projections is Borel}
\author{Dmitry Krokhalev \\ \href{mailto:dima.krohalev@bk.ru}{dima.krohalev@bk.ru}}

\begin{abstract}
We construct a bounded Borel set in the plane whose orthogonal projection
onto every line is non-Borel. This gives a negative answer to a question
of Thierry De Pauw.
\end{abstract}

\maketitle

\section{Introduction}

The orthogonal projection of a Borel subset of the plane need not be
Borel. Such projections are analytic sets, and the existence of analytic
non-Borel sets goes back to Suslin~\cite{Su}. De Pauw asked whether every
Borel subset of the plane nevertheless has a Borel orthogonal projection
in at least one direction; see Tristant~\cite{Tr} for this question and
related constructions. We answer the question in the negative.

\begin{theorem}\label{thm:main}
There exists a bounded Borel set $\cB\subset\R^2$ whose orthogonal
projection onto every line in the plane is not a Borel set.
\end{theorem}

Our proof begins with a Borel set whose projection onto one coordinate
is non-Borel. We encode this projection simultaneously in a family of
directions. To keep the copies from interfering, we work in a second
plane: a point $(x,y)$ represents the line $s=x+ty$, so projection in
direction $t$ becomes intersection with the vertical line at $t$.
We construct a set $\Sigma$ with Cantor vertical sections such that
through each point of $\Sigma$ there passes a Cantor family of lines meeting
$\Sigma$ at that point only. A decreasing function with carefully chosen
jumps supplies these lines. The construction first gives non-Borel
projections for all irrational $t\in(2,3)$. Finitely many rotated and
translated copies then give the result in every direction.

In Section~\ref{sec:framework} we provide the descriptive-set-theoretic facts that we will further need and
the geometric formulation of the main question. Sections~\ref{sec:function}
and~\ref{sec:slopes} construct the required families of lines.
The non-Borel set is encoded in the proof of the theorem in
Section~\ref{sec:main-proof}. The finite geometric construction and the
choice of parameters are given in the appendices.

\subsection*{Acknowledgments}
The author is grateful to Ioann Vasilyev for motivation, for introducing
him to the problem, and for helpful corrections and guidance.

\subsection*{AI use}
The author acknowledges the use of Claude Opus 5.5 in a guided way. The author is responsible for the content of the paper. 

\section{The encoding and the geometric formulation}\label{sec:framework}

\subsection{Two facts about Borel sets}
A Polish space is a separable completely metrizable topological space.
We use Polish spaces $\R$, $\R^2$, the Cantor space $\Cs$ of binary sequences with
the product topology, and their finite products. By a \emph{Cantor set}
we mean a space homeomorphic to $\Cs$; subsets of $\R$ are given the
subspace topology. A subset of a Polish space is called \emph{analytic} if it is a
continuous image of a Borel subset of a Polish space.

We will further need the following two classical facts.

\begin{factA}[Suslin; see~{\cite[Thm.~14.2, Ex.~14.3]{Ke}}]
For every Cantor set $D\subset\R$, there are a non-Borel set $A\subset D$
and a Borel set $F\subset D\times\Cs$ such that
\[
A=\{d\in D:\text{there exists }\omega\in\Cs\text{ with }(d,\omega)\in F\}.
\]
\end{factA}

Indeed, $D$ contains an analytic non-Borel set $A$, and $A$ is the
projection of a closed subset of $D\times\N^{\N}$. The Baire space
$\N^{\N}$ is homeomorphic to the $G_\delta$ subset of $\Cs$ consisting
of sequences with infinitely many ones. Under this homeomorphism the
closed set becomes a Borel subset of $D\times\Cs$.

\begin{factB}[Lusin--Suslin; see~{\cite[Cor.~15.2]{Ke}}, {\cite[\S\,4.5]{Sr}}]
Let $X,Y$ be Polish spaces, let $S\subset X$ be Borel, and let
$f\colon S\to Y$ be a Borel injective map. Then $f(S)$ is Borel.
\end{factB}

For the classical theory of analytic sets, see also~\cite{Lu}.

\subsection{Reduction to an arc}
Write $e_\theta=(\cos\theta,\sin\theta)$ and
$p_\theta(z)=\ip{z}{e_\theta}$. After identifying the target line with
$\R$, orthogonal projection coincides with $p_\theta$. Directions are considered
modulo $\pi$. For nonvertical directions it is convenient to use
\[
\pi_t(x,y)=x+ty.
\]
If $t=\tan\theta$, then $p_\theta=\cos\theta\,\pi_t$; multiplication
by a nonzero constant preserves Borelness.

The following finite construction allows us to restrict attention to a
small interval of parameters. Its proof is in Appendix~\ref{app:directions}.

\begin{lemma}\label{lem:arc}
Let $B_0\subset\R^2$ be a bounded Borel set, let
$J\subset\R/\pi\Z$ be a nonempty open arc, and let $E\subset J$ be
countable. If $p_\theta(B_0)$ is non-Borel for every
$\theta\in J\setminus E$, then there exists a set $\cB$ as in
Theorem~\ref{thm:main}.
\end{lemma}

We shall construct $B_0$ such that $\pi_t(B_0)$ is non-Borel for every
$t\in T$, where
\[
T=(2,3)\setminus\Q.
\]
Lemma~\ref{lem:arc} then applies with
$J=\{\arctan t:2<t<3\}$ and
$E=\{\arctan t:t\in\Q\cap(2,3)\}$.

\subsection{Private slopes}
To a point $b=(x,y)$ associate the line
\[
\ell_b=\{(t,s)\in\R^2:s=x+ty\}.
\]
Its intersection with the vertical line at $t$ has ordinate $\pi_t(b)$.
The line of slope $m$ through $(t_0,s_0)$ corresponds to
$b=(s_0-mt_0,m)$.

\begin{definition}
For $\sigma\in\Sigma\subset\R^2$, a number $m$ is called a
\emph{private slope} of $\sigma$ if the line of slope $m$ through
$\sigma$ meets $\Sigma$ only at $\sigma$.
\end{definition}

We will construct a Cantor set $D$, a bounded Borel function
$u\colon T\times D\to\R$ injective in the second variable, and a Borel map
\[
\mu\colon T\times D\times\Cs\longrightarrow[\tfrac12,1]
\]
injective in the third variable, such that, for all $t\in T$, $d\in D$,
and $\omega\in\Cs$, the value $\mu(t,d,\omega)$ is a private slope
of $(t,u(t,d))$ in
\[
\Sigma=\{(t,u(t,d)):t\in T,\ d\in D\}.
\]
The variable $d$ will carry the set $A$ from Fact~A; the variable
$\omega$ will carry its Borel representation $F$. Private slopes
ensure that a line used for one point of $\Sigma$ does not add a point
to any other vertical section of $\Sigma$.

\begin{figure}[ht]
\centering
\begin{tikzpicture}[scale=1.1,>=stealth]
  \draw[->] (0,0) -- (6.4,0) node[below] {$t$};
  \draw[->] (0,0) -- (0,3.4) node[left] {$s$};
  \foreach \x/\o/\n in {1/0.95/t_1,3/1.0/t_0,5/0.6/t_2} {
    \draw[dashed,gray] (\x,0) -- (\x,3.2);
    \node[below] at (\x,0) {$\n$};
    \foreach \d in {0,0.06,0.15,0.21,0.9,0.96,1.05,1.11}
      \fill (\x,\o+\d) circle (0.9pt);
  }
  \draw[thick] (0.2,1.21) -- (5.8,2.89);
  \fill[white] (3,2.05) circle (2.2pt);
  \draw (3,2.05) circle (2.2pt);
  \node[above left] at (3,2.05) {$\sigma$};
  \node[right] at (5.05,1.4) {\small$\Sigma\cap\{t=t_2\}$};
\end{tikzpicture}
\caption{A line with a private slope through $\sigma$.
Only three vertical sections of $\Sigma$ are shown schematically.}
\label{fig:windows}
\end{figure}
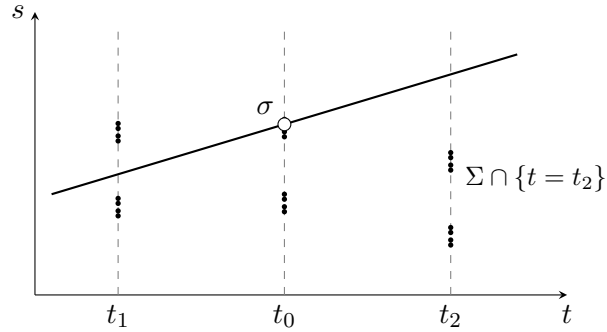

\section{The Cantor set and the jump function}\label{sec:function}

\subsection{Parameters and their roles}
We assign a positive weight to each rational number in
$\cR=\Q\cap(1,4)$. Choose integers
$1=Q_1<Q_2<\cdots$, with $Q_2=2$. A reduced fraction $r=p/q\in\cR$ by definition
has \emph{level} $i$ if $Q_i\le q<Q_{i+1}$. Let $n_i$ be the number
of fractions of level $i$, and give each of them weight $w_r=w_i$.
For a fixed $j$, fractions of levels at most $j$ are called heavy;
the others are called light. Set
\[
\lambda_j=\sum_{i>j}w_i n_i,\qquad L_0=1,\qquad
\ell_j=L_{j-1}/6.
\]

The numbers $\delta_j$ will be the digit sizes of $D$. The number
$\lambda_j$ controls the total light weight. The integers $Q_i$ ensure
that fractions of level $j+1$ occur in sufficiently short intervals.
Finally, $L_j$ is the length of a slope interval on which the $j$-th
avoidance condition holds. We require, for every $j\ge1$,
\begin{enumerate}[label=(P\arabic*)]
\item\label{P1}
$1=w_1>w_2>\cdots$, $\delta_1=1/100$, and
$\delta_{j+1}\le\min(\delta_j,w_{j+1})/100$;
\item\label{P2}
$\lambda_j\le\delta_j/100$ and
$\lambda_{j+1}\le\delta_j\ell_j/100$;
\item\label{P3}
$w_{j+1}\le\delta_j\ell_j/16$;
\item\label{P4}
$Q_{j+2}\ge200/(\delta_j\ell_j)$;
\item\label{P5}
$L_j=\min(\ell_j/2,w_{j+1}/(3\delta_j))$.
\end{enumerate}
Appendix~\ref{app:params} gives a recursive choice satisfying all these
conditions. We shall use the following consequences:
\begin{align}
&w_r\ge w_j\ge100\delta_j, \; \text{for every heavy fraction }r,\label{eq:heavy}\\
&4\delta_{j+1}\le\delta_j\ell_j/400, \; L_j\le\ell_j/48.\label{eq:small}
\end{align}
The first estimate follows from $\delta_j\le w_j/100$.
For the second line, use first $\delta_{j+1}\le w_{j+1}/100$ and~\ref{P3}, and
then~\ref{P5}.

\subsection{The digit scales}
Define
\[
D=\Bigl\{x\in \R: \; \exists (a_j)_{j\geq 1}\in \{0,1\}^{\N}, x=\sum_{j\ge1}a_j\delta_j\Bigr\}.
\]
Since $\sum_{i>j}\delta_i<2\delta_{j+1}<\delta_j$, distinct digit
sequences give distinct points. The continuous map from $\Cs$ onto
$D$ is therefore a homeomorphism. In particular, $D$ is a Cantor set.
For each $j\ge1$, the digit map
\[
a_j\colon D\longrightarrow\{0,1\},\qquad d\longmapsto a_j(d),
\]
is continuous when $D$ carries the subspace topology inherited from
$\R$ and $\{0,1\}$ carries the discrete topology. Indeed, under the
homeomorphism between $D$ and $\Cs$, this map is the projection onto
the $j$-th coordinate.

Set
\[
\alpha_j:=\delta_j-2\delta_{j+1},\qquad
\beta_j:=\delta_j+2\delta_{j+1},\qquad
H_j:=[\alpha_j,\beta_j].
\]
Thus $\alpha_j\ge0.98\delta_j$, $\beta_j<2\delta_j$, and
$\beta_j-\alpha_j=4\delta_{j+1}$. If $j$ is the first digit at which
$d,d_0\in D$ differ, then
\begin{equation}\label{eq:digits}
d_0-d\in
\begin{cases}
-H_j,&a_j(d_0)=0,\\
 H_j,&a_j(d_0)=1.
\end{cases}
\end{equation}
Indeed, the first differing digit contributes $\pm\delta_j$, and the
absolute value of the remaining sum is less than $2\delta_{j+1}$.

\subsection{The jump function and its increments}
Define, for $t\in(1,4)$,
\[
g(t)=-\sum_{r\in\cR,\ r<t}w_r,\qquad
u(t,d)=d+g(t)\quad(t\in T,\ d\in D).
\]
The weights have finite total sum by~\ref{P2}. Hence $g$ is a bounded
Borel function. It is non-increasing, has a downward jump of size $w_r$
at each rational $r$, and is continuous at each irrational point.
The last assertion follows by first retaining finitely many jumps and
then making the remaining total weight arbitrarily small.
Each vertical section $u(t,D)=D+g(t)$ is a translate of $D$.

Let $W(a,b)$ be the sum of the weights strictly between $a$ and $b$,
in either order. Let $W[a,b]$ include the endpoints as well.
For a fixed $t_0\in T$ and $m\in[\tfrac12,1]$, put
\begin{gather*}
\Fm(t)=m|t-t_0|+W(t_0,t)\quad(t\in(1,4)\setminus\Q),\\
\Fl(r)=m|r-t_0|+W(t_0,r),\qquad
\Fr(r)=\Fl(r)+w_r\quad(r\in\cR).
\end{gather*}
Here $\Fm(t_0)=0$. We suppress $t_0$ in this notation.
The values $\Fl(r)$ and $\Fr(r)$ are the limits through irrational
points approaching $r$ from the side nearer to $t_0$ and from the
farther side, respectively; the signs do not refer to left and right.
For irrational $t$,
\begin{equation}\label{eq:H}
g(t)-g(t_0)-m(t-t_0)=
\begin{cases}
-\Fm(t),&t>t_0,\\
 \Fm(t),&t<t_0.
\end{cases}
\end{equation}

On either side of $t_0$, $\Fm$ strictly increases with distance from
$t_0$. More precisely, if $t,r,t'$ lie on the same side,
$t,t'$ are irrational, and
$|t-t_0|<|r-t_0|<|t'-t_0|$, then
\begin{equation}\label{eq:mono}
\Fm(t)<\Fl(r)<\Fr(r)<\Fm(t'),
\qquad \Fm(t)\ge\tfrac12|t-t_0|.
\end{equation}
The successive differences in the chain are
$m|r-t|+W(t,r)$, $w_r$, and $m|t'-r|+W(r,t')$, all positive.

\section{Construction of private slopes}\label{sec:slopes}

\subsection{An avoidance condition at each digit scale}
Fix $t_0\in T$ and $d_0\in D$. At level $j$, the relevant side of
$t_0$ is $(t_0,4)$ if $a_j(d_0)=0$, and $(1,t_0)$ if
$a_j(d_0)=1$. We say that a slope $m$ \emph{jumps over level $j$}
if there is a rational $r$ on this side such that
\begin{equation}\label{eq:skip}
\Fl(r)<\alpha_j,\qquad \Fr(r)>\beta_j.
\end{equation}
Thus the jump at $r$ covers the whole interval $H_j$.

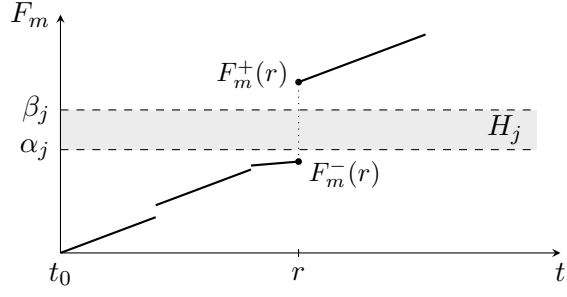
\begin{figure}[ht]
\centering
\begin{tikzpicture}[scale=1.05,>=stealth]
  \fill[gray!15] (0,1.3) rectangle (6,1.8);
  \draw[->] (0,0) -- (6.3,0) node[below] {$t$};
  \draw[->] (0,0) -- (0,3) node[left] {$\Fm$};
  \draw[dashed] (0,1.3) -- (6,1.3);
  \draw[dashed] (0,1.8) -- (6,1.8);
  \node[left] at (0,1.3) {$\alpha_j$};
  \node[left] at (0,1.8) {$\beta_j$};
  \node at (5.6,1.55) {$H_j$};
  \draw[thick] (0,0) -- (1.2,0.45) (1.2,0.6) -- (2.4,1.05);
  \draw[thick] (2.4,1.1) -- (3,1.15);
  \draw[thick] (3,2.15) -- (4.6,2.75);
  \draw[dotted] (3,1.15) -- (3,2.15);
  \fill (3,1.15) circle (1.2pt);
  \fill (3,2.15) circle (1.2pt);
  \node[right] at (3,1.0) {\small$\Fl(r)$};
  \node[left] at (3,2.25) {\small$\Fr(r)$};
  \draw (3,0.05) -- (3,-0.05) node[below] {$r$};
  \node[below] at (0,0) {$t_0$};
\end{tikzpicture}
\caption{A jump covering $H_j$, shown on the right of $t_0$.
The graph is schematic.}\label{fig:jump}
\end{figure}

\begin{lemma}\label{lem:private}
If $m\in[\tfrac12,1]$ jumps over every level for $(t_0,d_0)$, then
$m$ is a private slope of $(t_0,u(t_0,d_0))$ in $\Sigma$.
\end{lemma}

\begin{proof}
An intersection of the line in question with $(t,u(t,d))\in\Sigma$
requires
\begin{equation}\label{eq:intersection}
g(t)-g(t_0)-m(t-t_0)=d_0-d.
\end{equation}
For $t=t_0$, this gives $d=d_0$. If $t\ne t_0$ and $d=d_0$,
the left-hand side is nonzero by~\eqref{eq:H} and~\eqref{eq:mono}.
It remains to consider $t\ne t_0$, $d\ne d_0$.

Let $j$ be the first differing digit. Equations~\eqref{eq:digits}
and~\eqref{eq:H} show that~\eqref{eq:intersection} can hold only
on the relevant side of $t_0$, and only if $\Fm(t)\in H_j$.
Choose $r$ as in~\eqref{eq:skip}. Since $t$ is irrational, it lies
either nearer to $t_0$ than $r$ or farther away. In the first case,
\eqref{eq:mono} gives $\Fm(t)<\Fl(r)<\alpha_j$; in the second,
it gives $\Fm(t)>\Fr(r)>\beta_j$. Thus~\eqref{eq:intersection}
has no solution other than $(t,d)=(t_0,d_0)$.
\end{proof}

\subsection{The jumping lemma}
The next lemma allows us to impose one more avoidance condition inside
any slope interval of the prescribed length.

\begin{lemma}[Jumping lemma]\label{lem:jump}
Let $t_0\in T$, let $j\ge1$, and let $S$ be either $(t_0,4)$ or
$(1,t_0)$. Every closed interval $I\subset[\tfrac12,1]$ of length
$\ell_j$ contains a closed interval $I'$ of length $L_j$ for which
there is a rational $r\in S$ satisfying
\[
\Fl(r)<\alpha_j,\qquad \Fr(r)>\beta_j
\quad\text{for every }m\in I'.
\]
\end{lemma}

\begin{proof}
We prove the assertion for $S=(t_0,4)$. Reflection $t\mapsto5-t$
preserves the levels and weights of rational numbers and gives the
other case. Throughout the proof, $t$ and $t'$ denote irrational points.
Abbreviate
\[
\alpha=\alpha_j,\quad\beta=\beta_j,\quad\delta=\delta_j,
\quad\ell=\ell_j,\quad w=w_{j+1}.
\]

\emph{1. Separate a heavy jump from the light jumps.}
Let $r_0$ be the nearest heavy fraction to the right of $t_0$.
It exists because the fraction $3/1$ is heavy. There are only light fractions between
$t_0$ and $r_0$, so
\begin{equation}\label{eq:light}
W(t_0,r_0)\le\lambda_j\le\delta/100.
\end{equation}
The affine increasing function
\[
f(m)=m(r_0-t_0)+W(t_0,r_0)=\Fl(r_0)
\]
is the limit of $\Fm(t)$ as $t\uparrow r_0$.

If $\{m\in I:f(m)<\alpha\}$ has length at least $\ell/2$, choose
a closed interval $I'$ of length $L_j$ in this set; this is possible
because $L_j\le\ell/48$. Taking $r=r_0$ gives
$\Fl(r)<\alpha$, while~\eqref{eq:heavy} gives
$\Fr(r)\ge w_j\ge100\delta>\beta$. This proves the lemma in this case.

Otherwise there is a closed interval
$[m_B,m_B+\ell/2]\subset I$ on which $f\ge\alpha$.
By~\eqref{eq:light},
\[
r_0-t_0\ge\alpha-\delta/100\ge\delta/2.
\]
Put $m_1=m_B+\ell/8$ and $m_2=m_B+\ell/4$. Then
\[
f(m_1)\ge\alpha+\ell\delta/16>\alpha,
\qquad f(m_2)\ge\alpha+\ell\delta/8>\beta,
\]
where $\beta-\alpha\le\ell\delta/400$ by~\eqref{eq:small}.

\emph{2. Find a fraction of level $j+1$ between two crossing points.}
Define
\begin{align*}
\tau_1&=\inf\{t\in(t_0,r_0)\setminus\Q:F_{m_1}(t)\ge\alpha\},\\
\tau_2&=\inf\{t\in(t_0,r_0)\setminus\Q:F_{m_2}(t)\ge\beta\}.
\end{align*}
Continuity of $g$ at $t_0$ gives $\Fm(t)\to0$ as $t\downarrow t_0$.
The inequalities for $f(m_1)$ and $f(m_2)$ therefore place both
crossing points strictly between $t_0$ and $r_0$.
Monotonicity implies that $F_{m_1}<\alpha$ before $\tau_1$ and
$F_{m_1}\ge\alpha$ after it, and similarly for $F_{m_2}$ and $\beta$.

On $(t_0,r_0)$, we have
$\tfrac12(t-t_0)\le\Fm(t)\le(t-t_0)+\delta/100$.
Applying these bounds at the two crossing points gives
\begin{equation}\label{eq:crossings}
\tau_1,\tau_2\ge t_0+\delta/2,
\qquad \tau_1\le t_0+2\alpha<t_0+2\delta.
\end{equation}
Moreover, for $t>\tau_1$,
\[
F_{m_2}(t)=F_{m_1}(t)+\tfrac\ell8(t-t_0)
\ge\alpha+\ell\delta/16>\beta.
\]
Thus $\tau_2\le\tau_1$.

For $t<\tau_2$ and $t'>\tau_1$, both in $(t_0,r_0)$, we have
\[
F_{m_1}(t)<\beta-\tfrac\ell8(t-t_0),
\qquad F_{m_1}(t')\ge\alpha.
\]
Subtract these inequalities and let $t\uparrow\tau_2$,
$t'\downarrow\tau_1$. The total weight is finite, and these one-sided
limits include the weights at both endpoints. Consequently,
\begin{equation}\label{eq:growth}
\begin{split}
m_1(\tau_1-\tau_2)+W[\tau_2,\tau_1]
&\ge\tfrac\ell8(\tau_2-t_0)-(\beta-\alpha)\\
&\ge\ell\delta/16-\ell\delta/400
\ge\ell\delta/32.
\end{split}
\end{equation}

The interval $[\tau_2,\tau_1]\subset(t_0,r_0)$ contains no heavy
fractions. We claim that it contains a fraction of level $j+1$.
If the first term on the left of~\eqref{eq:growth} is at least
$\ell\delta/64$, then, since $m_1\le1$,
\[
\tau_1-\tau_2\ge\ell\delta/64>2/Q_{j+2}.
\]
For an integer $Q\ge2$, every interval of length greater than $2/Q$ contains a point
$k/(Q-1)$, whose reduced denominator is less than $Q$.
Apply this with $Q=Q_{j+2}$. As there are no heavy fractions here,
the fraction obtained has level exactly $j+1$.
If the first term is smaller, then
$W[\tau_2,\tau_1]\ge\ell\delta/64$, whereas the total weight of
levels $j+2,j+3,\ldots$ is at most
$\lambda_{j+1}\le\ell\delta/100$. This again forces a fraction of
level $j+1$ in the interval. Choose one and call it $r$; its weight is $w$.

\emph{3. Vary the slope so that the jump at $r$ covers $H_j$.}
Write
\[
s=r-t_0\in[\delta/2,2\delta],\qquad
G(m)=ms+W(t_0,r)=\Fl(r).
\]
Since $\tau_2\le r\le\tau_1$, one-sided limits at $r$ give
\[
G(m_1)\le\alpha,\qquad G(m_2)+w\ge\beta.
\]
Using $w\le\ell\delta/16$ and $s\ge\delta/2$, we obtain
\begin{align*}
G(m_B)&=G(m_1)-s\ell/8
\le\alpha-\ell\delta/16<\beta-w,\\
G(m_B+3\ell/8)&=G(m_2)+s\ell/8
\ge\beta-w+\ell\delta/16\ge\beta>\alpha.
\end{align*}
Both tested slopes lie in $I$. Since $G$ is strictly increasing,
\[
\{m:\beta-w<G(m)<\alpha\}
\]
is an open interval contained in $I$. Its length is
\[
\frac{w-(\beta-\alpha)}s
=\frac{w-4\delta_{j+1}}s
\ge\frac{0.96w}{2\delta}
>\frac{w}{3\delta}\ge L_j.
\]
Here we used $\delta_{j+1}\le w/100$ and~\ref{P5}.
Choose a closed interval $I'$ of length $L_j$ inside this open interval.
For every $m\in I'$, we then have $\Fl(r)<\alpha$ and
$\Fr(r)=G(m)+w>\beta$, as required.
\end{proof}

\subsection{A Borel family of private slopes}
We now iterate the jumping lemma along a binary tree. At each step we
retain two separated slope intervals, so that all binary sequences give
different limiting slopes. A countable least-choice rule makes the
construction Borel in the base point.

\begin{lemma}\label{lem:select}
There exists a Borel map
$\mu\colon T\times D\times\Cs\to[\tfrac12,1]$, injective in the third variable,
such that every $\mu(t_0,d_0,\omega)$ is a private slope of
$(t_0,u(t_0,d_0))$ in $\Sigma$.
\end{lemma}

\begin{proof}
Enumerate all pairs $(r_k,J_k)$, where $r_k\in\cR$ and
$J_k\subset[\tfrac12,1]$ is a closed interval with rational endpoints.
For each $(t_0,d_0)$ we construct closed intervals $K_s$, indexed by
finite binary words $s$, with
\[
K_\varnothing=[\tfrac12,1],\qquad |K_s|=L_{|s|}/2.
\]
Every slope in $K_s$ will jump over levels $1,\ldots,|s|$.

Suppose $|s|=j-1$ and $K_s$ is constructed. Its left and right closed
thirds, denoted $K_s^0,K_s^1$, are separated and have length $\ell_j$.
For each $e\in\{0,1\}$ choose the least $k$ such that
\begin{enumerate}[label=(\roman*)]
\item $J_k\subset K_s^e$ and $|J_k|\ge L_j/2$;
\item $r_k$ is on the relevant side of $t_0$ at level $j$;
\item $\Fl(r_k)<\alpha_j$ and $\Fr(r_k)>\beta_j$ at both
endpoints $m$ of $J_k$.
\end{enumerate}
For fixed $t_0,r_k$, both expressions in (iii) are affine in $m$;
thus (iii) holds throughout $J_k$. The jumping lemma supplies a rational
$r$ and a closed slope interval of length $L_j$ inside $K_s^e$.
This interval contains a closed interval with rational endpoints of
length at least $L_j/2$, so an admissible $k$ exists.
Let $K_{se}$ be the left subinterval of $J_k$ of length $L_j/2$.
It retains all earlier avoidance conditions and adds the condition
at level $j$.

The intervals along every branch are nested, and
$L_j\le L_{j-1}/12$ by~\ref{P5}. Hence
\[
\{\mu(t_0,d_0,\omega)\}
=\bigcap_{j\ge1}K_{\omega_1\cdots\omega_j}
\]
defines a single slope. Different branches separate at their first
differing digit, so the map is injective in the third variable.
The limiting slope belongs to every interval on its branch and
therefore jumps over every level. Lemma~\ref{lem:private} makes it
private. Agreement of the first $j$ digits confines two limiting slopes
to one interval of length $L_j/2$; in particular, the map is continuous
in $\omega$ for fixed $(t_0,d_0)$.

It remains to check joint Borel measurability. For each fixed rational
$r$, the function $t_0\mapsto W(t_0,r)$ is a convergent sum of Borel
functions, and $d_0\mapsto a_j(d_0)$ is continuous.
Assume inductively that the endpoints of $K_s$ are Borel functions of
$(t_0,d_0)$. For each $k$, conditions (i)--(iii) then define a Borel
set of base points. The least admissible index is Borel: its level set
at $k$ is the admissibility set for $k$ with the finitely many earlier
admissibility sets removed. Thus the endpoints of each child interval
are Borel functions as well.
For fixed $j$, the left endpoint of $K_{\omega_1\cdots\omega_j}$ is
jointly Borel, since it is Borel on each of the finitely many binary
cylinders of length $j$. These endpoints converge to $\mu$, proving
the assertion.
\end{proof}

\section{Proof of the main theorem}\label{sec:main-proof}

\begin{proof}[Proof of Theorem~\ref{thm:main}]
Take the sets $A\subset D$ and $F\subset D\times\Cs$ from Fact~A
and the slope map $\mu$ from Lemma~\ref{lem:select}. Define
\[
\Gamma(t,d,\omega)=\bigl(u(t,d)-mt,m\bigr),
\qquad m=\mu(t,d,\omega),\qquad
B_0=\Gamma(T\times F).
\]
The point $\Gamma(t,d,\omega)$ represents the line of slope $m$
through $(t,u(t,d))$. Since $u$ is bounded, $2<t<3$, and
$\tfrac12\le m\le1$, the set $B_0$ is bounded.

We first verify that $B_0$ is Borel. The domain $T\times F$ is Borel
in the Polish space $\R\times D\times\Cs$, and $\Gamma$ is Borel.
If two triples have the same image, they represent the same line.
As this line has a private slope at its first base point, their points
in $\Sigma$ coincide. Thus their $t$-coordinates coincide, and so do
their $d$-coordinates because $u(t,\cdot)$ is injective. Equality of
the slopes and injectivity of $\mu$ in the third variable give equality
of the last coordinates. Hence $\Gamma$ is injective, and Fact~B applies.

For each $t\in T$, the compact set $u(t,D)=D+g(t)$ recovers exactly
the encoded copy of $A$:
\begin{equation}\label{eq:window}
\pi_t(B_0)\cap u(t,D)=u(t,A).
\end{equation}
Indeed, if $d\in A$, some $(d,\omega)$ belongs to $F$, and the line
represented by $\Gamma(t,d,\omega)$ passes through $(t,u(t,d))$.
This proves inclusion of the right-hand side in the left-hand side.
Conversely, if a line represented by
$\Gamma(t',d',\omega')$, with $(d',\omega')\in F$, passes through
$(t,u(t,d))\in\Sigma$, its private-slope property forces
$(t,u(t,d))=(t',u(t',d'))$. Hence $t=t'$, $d=d'\in A$, proving
the other inclusion.

The right-hand side of~\eqref{eq:window} is a translate of the
non-Borel set $A$, and is therefore non-Borel. Since the intersection
of a Borel set with the compact set $u(t,D)$ would be Borel,
$\pi_t(B_0)$ is non-Borel for every $t\in T$.
Apply Lemma~\ref{lem:arc} to the arc
$\{\arctan t:2<t<3\}$ with its countable set of rational parameters
removed. The resulting finite union of rotated and translated copies
of $B_0$ is bounded and Borel, and all its orthogonal projections are
non-Borel.
\end{proof}

\appendix
\section{From an arc to every direction}\label{app:directions}

\begin{proof}[Proof of Lemma~\ref{lem:arc}]
Choose $\rho>0$ so that $B_0$ lies in the disc of radius $\rho$
centred at the origin. If $R_\alpha$ denotes rotation through $\alpha$,
then
\[
p_\theta(R_\alpha B_0+c)
=p_{\theta-\alpha}(B_0)+p_\theta(c).
\]
Thus this projection is non-Borel whenever
$\theta\in(J\setminus E)+\alpha$.

Choose $N$ so that $N|J|\ge6\pi$. Perturb the equally spaced angles
to obtain
\[
\alpha_k=k\pi/N+\varepsilon_k,\qquad
|\varepsilon_k|<\pi/(4N),\qquad k=1,\ldots,N,
\]
with $\alpha_k-\alpha_l\notin E-E$ modulo $\pi$ for $k\ne l$.
These perturbations can be chosen successively, since at each step
only countably many values are excluded.
Every direction belongs to at least four arcs $J+\alpha_k$.
Indeed, an arc of length $|J|$ contains at least
$N|J|/\pi-2\ge4$ of the perturbed angles.
The sets $E+\alpha_k$ are pairwise disjoint, so at most one of these
indices is exceptional in any given direction.

Next choose distinct points $c_1,\ldots,c_N\in\R^2$ so that the vectors
$c_k-c_l$, $k<l$, are pairwise nonparallel. An inductive choice is
possible by avoiding finitely many lines at each step. Put
\[
B_k=R_{\alpha_k}B_0+\lambda c_k,
\qquad \cB=\bigcup_{k=1}^N B_k,
\]
where $\lambda$ will be chosen large. The projection $p_\theta(B_k)$
lies in the closed interval $\Delta_k(\theta)$ of radius $\rho$
centred at $\lambda p_\theta(c_k)$. Two such intervals can intersect
only if
\[
\theta\in O_{kl}
=\{\theta:|p_\theta(c_k-c_l)|\le2\rho/\lambda\}.
\]
As $\lambda\to\infty$, each $O_{kl}$ shrinks to the direction
orthogonal to $c_k-c_l$. These directions are distinct. Since there
are finitely many pairs, we may fix $\lambda$ so that the sets
$O_{kl}$ are pairwise disjoint.

Fix any direction $\theta$. At least four indices satisfy
$\theta\in J+\alpha_k$. At most one is exceptional, and at most two
belong to a pair of overlapping projection intervals. There is
therefore an index $k$ for which $\theta-\alpha_k\in J\setminus E$
and $\Delta_k(\theta)$ is disjoint from all other projection intervals.
Consequently,
\[
p_\theta(\cB)\cap\Delta_k(\theta)
=p_\theta(B_k)
=p_{\theta-\alpha_k}(B_0)+\lambda p_\theta(c_k)
\]
is non-Borel. The interval $\Delta_k(\theta)$ is closed, so
$p_\theta(\cB)$ is non-Borel as well. The set $\cB$ is a finite
union of bounded Borel sets, which completes the construction.
\end{proof}

\section{Choice of parameters}\label{app:params}

Set $Q_1=1$, $Q_2=2$, $w_1=1$, $\delta_1=1/100$, and $L_0=1$.
Level $1$ consists of the fractions $2$ and $3$. At step $j\ge1$,
suppose that
\[
Q_1,\ldots,Q_{j+1},\quad w_1,\ldots,w_j,\quad
\delta_1,\ldots,\delta_j,\quad L_0,\ldots,L_{j-1}
\]
have been chosen. Then $\ell_j=L_{j-1}/6$ is known, as is
\[
\varepsilon_j=\frac1{200}\min_{1\le k\le j}\delta_k\ell_k.
\]
The sequence $\varepsilon_j$ is non-increasing. Make the following
choices, in order:
\begin{enumerate}[label=(\roman*)]
\item Choose an integer $Q_{j+2}>Q_{j+1}$ satisfying~\ref{P4}.
This determines the finite number $n_{j+1}$.
\item Choose $0<w_{j+1}<w_j$ satisfying~\ref{P3} and
$w_{j+1}n_{j+1}\le2^{-j}\varepsilon_j$.
\item Set $\delta_{j+1}=\min(\delta_j,w_{j+1})/100$ and choose
$L_j$ by~\ref{P5}.
\end{enumerate}
There is no circularity: each bound in (i)--(iii) uses quantities
already chosen. Conditions~\ref{P1}, \ref{P3}, \ref{P4}, and~\ref{P5}
hold by construction.

To verify~\ref{P2}, use
$w_i n_i\le2^{-(i-1)}\varepsilon_{i-1}$ for $i\ge2$. Then
\[
\lambda_j\le\varepsilon_j\sum_{i>j}2^{-(i-1)}
\le\varepsilon_j\le\frac{\delta_j\ell_j}{200}
\le\frac{\delta_j}{100},
\]
where $\ell_j\le1$, and
\[
\lambda_{j+1}\le\varepsilon_{j+1}
\le\frac{\delta_j\ell_j}{200}
\le\frac{\delta_j\ell_j}{100}.
\]
In particular, $\sum_iw_i n_i<\infty$, as required for the jump function.

\end{document}